\documentclass[a4paper,11pt]{amsart}
\usepackage{amsmath,amsthm,amsfonts,amssymb,color}
 \usepackage{a4wide}
\usepackage[dvips]{graphicx}
\usepackage{ulem}
\theoremstyle{plain}
\newtheorem{thm}{Theorem}[section]
\newtheorem{prop}[thm]{Proposition}
\newtheorem{lem}[thm]{Lemma}
\newtheorem{defn}[thm]{Definition}
\newtheorem{cor}[thm]{Corollary}

\theoremstyle{remark}

\newtheorem{rem}[thm]{Remark}

\newcommand{\N}{\mathbb{N}}
\newcommand{\Q}{\mathbb{Q}}
\newcommand{\R}{\mathbb{R}}
\newcommand{\Z}{\mathbb{Z}}
\newcommand{\C}{\mathbb{C}}

\title{Spherical designs and modular forms of the $D_4$ lattice}

\author[M.~Hirao]{Masatake Hirao}
\address[M.~Hirao]{Department of Information Science and Technology\\
	Aichi Prefectural University\\
	Nagakute-city, Aichi, 480-1198\\
	Japan}
\email{hirao@ist.aichi-pu.ac.jp}

\author[H.~Nozaki]{Hiroshi Nozaki}
\address[H.~Nozaki]{Department of Mathematics Education\\ 
	Aichi University of Education\\
	1 Hirosawa, Igaya-cho, Kariya, Aichi 448-8542\\
	Japan}
\email{hnozaki@auecc.aichi-edu.ac.jp}

\author[K.~Tasaka]{Koji Tasaka}
\address[K.~Tasaka]{Department of Information Science and Technology\\
	Aichi Prefectural University\\
	Nagakute-city, Aichi, 480-1198\\
	Japan}
\email{tasaka@ist.aichi-pu.ac.jp}

\subjclass{05B30, 11P21, 11F30, 11F33}
\keywords{Spherical designs of harmonic index, the $D_4$ root lattice/system, weighted theta functions, non-vanishing and congruences of the Fourier coefficients of cusp forms of level 2}
\begin{document}

\maketitle

\begin{abstract}
In this paper, we study shells of the $D_4$ lattice with a {slight generalization} of spherical $t$-designs due to Delsarte-Goethals-Seidel, namely, the spherical design of harmonic index $T$ (spherical $T$-design for short) introduced by Delsarte-Seidel.
We first observe that{, for any positive integer $m$,} the $2m$-shell of $D_4$ is an antipodal spherical $\{10,4,2\}$-design on the three dimensional sphere.
We then prove that the $2$-shell, which is the $D_4$ root system, is a tight $\{10,4,2\}$-design, using the linear programming method.
The uniqueness of the $D_4$ root system as an antipodal spherical $\{10,4,2\}$-design with 24 points is shown.
We give two applications of the uniqueness: a decomposition of the shells of the $D_4$ lattice in terms of orthogonal transformations of the $D_4$ root system, and the uniqueness of the $D_4$ lattice as an even integral lattice of level 2 in the four dimensional Euclidean space.
We also reveal a connection between the harmonic strength of the shells of the $D_4$ lattice and non-vanishing of the Fourier coefficients of a certain newform of level 2.
Motivated by this, congruence relations for the Fourier coefficients are discussed.
\end{abstract}

%%%%%%%%%%%%%%%%%%%%%%%%%%%%%%%%%%
\section{Introduction}

For a positive integer $t$, a finite nonempty subset $X$ of the unit sphere $\mathbb{S}^{d-1}$ in the $d$-dimensional Euclidean space $\R^d$ is called a {\itshape spherical $t$-design} if 
\[\frac{1}{|X|} \sum_{\boldsymbol{x}\in X} F(\boldsymbol{x}) = \frac{\int_{\mathbb{S}^{d-1}}F(\xi)d\sigma(\xi)}{\int_{\mathbb{S}^{d-1}}d\sigma(\xi)}\]
holds for any real polynomial $F(\boldsymbol{x})=F(x_1,\ldots,x_d)$ of degree $\le t$.
Here the right-hand side is the usual surface integral over $\mathbb{S}^{d-1}$.
It is convenient to use the equivalent condition that 
$X\subset \mathbb{S}^{d-1}$ is a spherical $t$-design if and only if
\[
\sum_{\boldsymbol{x} \in X} P(\boldsymbol{x}) = 0, \quad \forall P \in {\rm Harm}_\ell (\R^d), \quad \forall \ell \in \{1,2,\ldots,t\},
\]
where ${\rm Harm}_{\ell}(\R^d)$ denotes the $\R$-vector space of real 
homogeneous harmonic polynomials (see Section 2 for the definition) of degree exactly $\ell$ in $d$ variables.
%\[ \sum_{\boldsymbol{x}\in X} P(\boldsymbol{x})=0\]
%holds for all real homogeneous harmonic polynomials $P(\boldsymbol{x})$, i.e.~polynomials in $\R[x_1,\ldots,x_d]$ annihilated by the Laplacian operator (see Section 2), of degree $\le t$.
%%

The concept of spherical designs was first introduced by Delsarte-Goethals-Seidel \cite{DelsarteGoethalsSeidel77}.
For any spherical $t$-design $X\subset \mathbb{S}^{d-1}$, they proved a so-called Fisher type bound $|X|\ge b_{d,t}$, where $b_{d,t}= \binom{d+e-1}{e}+\binom{d+e-2}{e-1}$ if $t=2e$ and $b_{d,t}= 2\binom{d+e-1}{e}$ if $t=2e+1$.
If a spherical $t$-design $X\subset \mathbb{S}^{d-1}$ satisfies $|X|=b_{d,t}$, it is said to be {\itshape tight}. 
Since tight spherical $t$-designs have good extremal properties, their classifications have been studied by many people.
For these studies, we refer to \cite{BannaiBannai09} and references therein.

For a positive integer $m$, the $m$-shell of a lattice is the set of lattice points on the sphere with $\sqrt{m}$ radius.
These finite sets have been studied from the design theoretical viewpoint in connection with modular forms, in particular, weighted theta functions.
In this paper, we wish to explicate the shells of the $D_4$ lattice, an even integral lattice in $\R^4$, using a slight generalization of spherical $t$-designs: a spherical design of harmonic index $T$ ({\itshape spherical $T$-design} for short).
Here, for a subset $T$ of $\N$, a non-empty finite subset $X$ of $\mathbb{S}^{d-1}$ is called a spherical $T$-design
if it holds that 
\[
\sum_{\boldsymbol{x} \in X} P(\boldsymbol{x}) = 0, \quad \forall P \in {\rm Harm}_\ell (\R^d), \quad \forall \ell \in T.
\]
%where ${\rm Harm}_{\ell}(\R^d)$ denotes the $\R$-vector space of real 
%homogeneous harmonic polynomials of degree exactly $\ell$ in $d$ variables.
This concept was first introduced by Delsarte-Seidel \cite{DS89} as a spherical analogue of the design in association schemes \cite[Section 3.4]{Delsarte:PHD}. % and it is well known that $X$ is a spherical $t$-design if it is a spherical $\{t,t-1,\ldots,2,1\}$-design.
A prototype of our work is due to Venkov \cite{Venkov84}; one of his results shows that any non-empty (normalized) $2m$-shell of an extremal even unimodular lattice in $\R^{24n} \ (n\ge1)$, including the Leech lattice, is a spherical $\{14,10,8,6,4,2\}$-design.
In his {work}, the theory of modular forms {on} the full modular group plays an important role.
{Since then there have been similar investigations} on several types of lattices (see e.g., \cite{BachocVenkov01,HarpePache05,HarpePacheVenkov06,Pache05}).

The $D_4$ lattice is a root lattice in $\R^4$ generated by all permutations of $(\pm1,\pm1,0,0)$ over $\Z$.
Its $2m$-shell, denoted by $(D_4)_{2m}$, becomes the set of integer solutions to the equation {$x_1^2+x_2^2+x_3^2+x_4^2=2m$}.
{We start by proving} that the normalized set $\frac{1}{\sqrt{2m}}(D_4)_{2m}$ on the unit sphere $\mathbb{S}^3$ is a spherical $\{10,4,2\}$-design for all $m\ge1$ (Proposition \ref{prop:D4 root}).
We indicate two proofs;
the first proof is based on the fact that the Weyl group $W(\mathbf{F}_4)$ of the root system $\mathbf{F}_4$ acts on the $D_4$ lattice, together with the formula for the harmonic Molien series of $W(\mathbf{F}_4)$;
the second proof uses the theory of modular forms of level 2 with weighted theta functions of the $D_4$ lattice.
As a special case, we see that the $D_4$ root system, which is the $2$-shell $(D_4)_2$, is an antipodal spherical $\{10,4,2\}$-design of $\mathbb{S}^3$ with 24 points.
A crucial discovery due to linear programming method is that the lower bound of the cardinality of such design is 24 (Theorem \ref{thm:upper_bound_design}).
{For an antipodal spherical $\{10,4,2\}$-design $X$ in $\mathbb{S}^3$, we say that $X$ is tight if $|X|=24$}. 
Then the $D_4$ root system becomes an example of a tight antipodal spherical $\{10,4,2\}$-design, while it is not a tight spherical $5$-design on $\mathbb{S}^3$ (since $b_{4,5}=20$).

More recently, the study of classification of tight spherical $T$-designs has attracted a lot of attention.
It was started in \cite{BannaiOkudaTagami15} for the case $T=\{t\}$.
The case $t = 4$ was investigated in \cite{OkudaYu16}. 
Zhu et al.~\cite{ZhuBannaiBannaiKimYu17} obtained the classification 
of tight spherical designs of harmonic index $6$ and $8$, 
as well as the asymptotic non-existence of tight spherical $\{2e\}$-designs for $e \geq 3$. 
They also studied the existence problem for tight spherical $T$-designs for some $T$, including the case $T = \{8, 4\}$. 
Our classification problem is based on the fact that the image of a spherical $T$-design under an orthogonal transformation (see Section 4 for the definition) is also a spherical $T$-design.
With this, we prove the uniqueness of the $D_4$ root system (Theorem \ref{thm:uniqueness}).

\begin{thm}\label{thm:classification}
Every antipodal spherical $\{10,4,2\}$-design on $\mathbb{S}^3$ with $24$ points is an orthogonal transformation of the normalized $D_4$ root system $\frac{1}{\sqrt{2}}(D_4)_{2}$.
\end{thm}

It is worth pointing out that the normalized $D_4$ root system is the first example such that it is not unique as a spherical $t$-design, but unique as an antipodal spherical $T$-design (see also Remark \ref{rem:unique}).
Theorem \ref{thm:classification} not only contributes to the study of classification of spherical designs, but also has two striking applications: a decomposition of the normzalized shells of the $D_4$ lattice in terms of the disjoint union of orthogonal transformations of the normalized $D_4$ root system (Theorem  \ref{thm:D_4-decomp}), and the uniqueness of the $D_4$ lattice as an even integral lattice of level 2 in $\R^4$ (Theorem \ref{thm:uniqueness_lattice}).

In connection with modular forms, given a finite subset $X\subset \mathbb{S}^{d-1}$, we deal with the problem of determining the maximal subset $T\subset\N$, called the {\itshape harmonic strength} of $X$, such that $X$ is a spherical $T$-design.
This problem is intimately related to the non-vanishing problem of the Fourier coefficients of cusp forms.
For the shells of the $D_4$ lattice, we obtaine the following result (Theorem \ref{thm:lehmer-type}).
\begin{thm}\label{thm:non-vanishing}
For $m\in \N$, the harmonic strength of $\frac{1}{\sqrt{2m}}(D_4)_{2m}$ contains $6$ if and only if the $m$th Fourier coefficient $\tau_2(m)$ of the unique normalized cusp form $\sum_{m\ge1} \tau_2(m)q^m:=\eta(z)^8\eta(2z)^8 \ (q:=e^{2\pi i z})$ of weight $8$ of level $2$ is zero, where $\eta(z):=q^{1/24}\prod_{n\ge1}(1-q^n)$ is the Dedekind eta function.
\end{thm}

This is analogues to the study of de la Harpe, Pache and Venkov \cite{HarpePache05,HarpePacheVenkov06}; as a prototype, it was known to Venkov for many years and stated in \cite[Proposition B]{Pache05} (see also \cite[Section 3.2]{BannaiBannai09}) that the Ramanujan $\tau$-function $\tau(m)$, defined by $\sum_{m\ge1} \tau(m) q^m :=\eta(z)^{24}$, vanishes if and only if the $2m$-shell of the $E_8$ lattice is a spherical $8$-design.
Note that $\tau(m)$ is believed to be non-zero for all $m\in \N$, according to Lehmer's conjecture \cite{Lehmer47}.
In our case, we may believe that $\tau_2(m)$ would never be 0 (similar to Lehmer's conjecture).
%{Moreover strongly, we expect that} the harmonic strength of the $2m$-shell of the $D_4$ lattice is $\{10,4,2\}$ for all $m\ge1$.
Along these lines, we prove congruence relations $\tau_2(p)\equiv p(p+1)\bmod \ell$ for $\ell \in\{3,5\}$ (Theorem \ref{thm:tau2_congruence}) which shows $\tau_2(p)\neq0$ for all prime $p\not\equiv-1\bmod 15$ (Corollary \ref{cor:non-vanishing}).
This congruence might not be new and can be deduced from results in the literature, e.g., \cite{BillereyMenares,DummiganFretwell,GabaPopa,KumarKumariMoreeSingh,Nikdelan}, but our proof may shed new light on this study.

{The organization of this paper is as follows. 
In Section 2, some basic materials, including the definitions of spherical designs and codes, are prepared.
In Section 3, using the linear programming method, we prove bounds for the cardinality of $(4,N,1/2)$ spherical codes and spherical $\{10,4,2\}$-designs on $\mathbb{S}^3 \ (d=4)$.
In Section 4, we recall some basic techniques from the theory of spherical designs and apply it to the shells of the $D_4$ lattice.
Section 5 is devoted to proving Theorem \ref{thm:classification}, the uniqueness of the $D_4$ root system as antipodal spherical $\{10,4,2\}$-designs on $\mathbb{S}^3$ with 24 points.
Section 6 and Section 7 discuss applications of the uniqueness theorem to the orthogonal decomposition of the shells of the $D_4$ lattice and to the uniqueness of the $D_4$ lattice as an even integral lattice of level 2 in $\R^4$, respectively. 
In Section 8, we prove Theorem \ref{thm:non-vanishing}.
}

%%%%%%%%%%%%%%%%%%%%%%%%%%%%%%%%%%
\section{Spherical code and design}

The concepts of spherical codes and spherical designs introduced by Delsarte-Goethals-Seidel \cite{DelsarteGoethalsSeidel77} apply for finite subsets of the unit sphere ${\mathbb S}^{d-1} := \{ \boldsymbol{x} = (x_1, \ldots, x_d) \in \R^{d} \mid \langle \boldsymbol{x} ,\boldsymbol{x}\rangle = 1 \}$ in the $d$-dimensional Euclidean space $\R^d$, where {$ \langle \boldsymbol{x}, \boldsymbol{y} \rangle := \sum_{i = 1}^{d} x_iy_i$ for $\boldsymbol{x},\boldsymbol{y}\in \R^d$}.
We recall their definitions, thereby also fixing some of our notation.

For a subset $X$ of $\mathbb{S}^{d-1}$, let us denote the set of inner products of two distinct points in $X$ by
\[A(X){:=}\{\langle \boldsymbol{x}, \boldsymbol{y}\rangle \mid  \boldsymbol{x}, \boldsymbol{y}\in X,\  \boldsymbol{x}\neq \boldsymbol{y}\} \subset [-1,1).\]
We denote by ${\rm Harm}_{\ell}(\R^d)$ the $\R$-vector space of real 
homogeneous harmonic polynomials of degree exactly $\ell$ in $d$ variables, namely, {polynomials} in $\R[x_1,\ldots,x_d]$ of homogeneous degree $\ell$ annihilated by the Laplacian operator 
\[{\Delta_d:=\sum_{j=1}^d \frac{\partial^2}{\partial x_j^2}}.\]
{It is well known (see Theorem 3.2 in \cite{DelsarteGoethalsSeidel77}) that
\[ \dim {\rm Harm}_\ell(\R^d)= \binom{d+\ell -1}{\ell} - \binom{d+\ell-3}{\ell -2}.\]
}

\begin{defn}\label{def:spherical design}
1) A set $X$ of $N$ points on $\mathbb{S}^{d-1}$ is called a $(d,N,a)$ spherical code if every element in $ A(X)$ is less than or equal to $a\in \R$.\\
2) Let $T$ be a subset of $\N$.
A non-empty finite subset $X$ of $\mathbb{S}^{d-1}$ is called a spherical design of harmonic index $T$ (spherical $T$-design for short)
if it holds that 
\[
\sum_{\boldsymbol{x} \in X} P(\boldsymbol{x}) = 0, \quad \forall P \in {\rm Harm}_\ell (\R^d), \quad \forall \ell \in T.
\]
\end{defn}

For $t\in \N$, a spherical $\{t,t-1,\ldots,2,1\}$-design is a spherical $t$-design as mentioned in Introduction (see also \cite{DelsarteGoethalsSeidel77} for the original definition).
The spherical $T$-design, which is a generalization of spherical $t$-designs, {was} first introduced by Delsarte-Seidel \cite{DS89} and its classification has recently been studied by Bannai-Okuda-Tagami \cite{BannaiOkudaTagami15}.

For a subset $X$ of $\R^d$ and a scalar $c\in \R$, we write $cX:=\{c\boldsymbol{x}\in \R^d\mid \boldsymbol{x}\in X\}$. % (called {\itshape normalized} if it lies in $\mathbb{S}^{d-1}$).
A set $X$ is said to be {\itshape antipodal} if we have $-X=X$. 
For an antipodal subset $X$ of $\R^d$, a subset $X'\subset X$ is called {\itshape a half set of $X$} if $X$ is a disjoint union of $X'$ and $-X'$; $X'\sqcup (-X')=X$.
For any antipodal {subset $X$ of $\mathbb{S}^{d-1}$ (note that $\boldsymbol{0}\not\in X$)}, a half set of $X$ always exists, but not unique.

\begin{lem}\label{lem:half set of antipodal T-design}
Let $X'$ be a half set of an antipodal subset $X\subset \mathbb{S}^{d-1}$.
If $X'$ is a spherical $T$-design, then $X$ is an antipodal spherical $T$-design.
On the other hand, if $X$ is an antipodal spherical $T$-design, then $X'$ is a spherical $T'$-design with $T'=\{2\ell \in 2\N\mid 2\ell\in T\}$, where $2\N$ is the set of positive even integers.
\end{lem}
\begin{proof}
Suppose that $X'$ is a spherical $T$-design.
Then, for $\ell\in T$ and $P\in {\rm Harm}_\ell(\R^d)$, one has
\[\sum_{\boldsymbol{x}\in X}P(\boldsymbol{x}) = \sum_{\boldsymbol{x}\in X'}P(\boldsymbol{x})+\sum_{\boldsymbol{x}\in -X'}P(\boldsymbol{x})=(1+(-1)^\ell)\sum_{\boldsymbol{x}\in X'}P(\boldsymbol{x})=0. \]
Hence, $X$ is an antipodal spherical $T$-design.
Now suppose that $X$ is an antipodal spherical $T$-design.
Then, for $\ell\in T$ even and $P\in {\rm Harm}_\ell(\R^d)$, we have $0=\sum_{\boldsymbol{x}\in X}P(\boldsymbol{x})=2\sum_{\boldsymbol{x}\in X'}P(\boldsymbol{x}) $, and hence, $X'$ is a spherical $T'$-design.
We complete the proof.
\end{proof}

We also notice that if $X$ is an antipodal spherical $T$-design, then $T$ contains all positive odd integers.
Since in this paper we only consider antipodal spherical $T$-designs,  we omit to write positive odd integers lying in $T$.

%%%%%%%%%%%%%%%%%%%%%%%%%%%%%%%%%%
\section{Linear programming bounds}
{The principle problem in the theory of spherical codes (resp.~a spherical design of harmonic index) is, for a fixed $d$ and $a$, to find a $(d,N,a)$ spherical code with maximum possible $N$ (resp.~for a fixed $d$ and $T$, to find a spherical $T$-design with minimum possible $N$).}
%The principal problem on a $(d,N,a)$ spherical code (resp.~a spherical $T$-design) is to maximize the cardinality $N$ for a given value of $a$ (resp.~to find minimal $N$ ). 
%In what follows, we give a solution to this problem for a $(4,N,1/2)$ spherical code and a spherical $\{10,4,2\}$-design, using the linear programming method. 
%Here 
The linear programming method, established by Delsarte-Goethals-Seidel \cite{DelsarteGoethalsSeidel77}, is a useful tool to provide upper (resp.~lower) bounds on the cardinality of a spherical code (resp.~design).
In this section, we describe and apply it to our cases: a spherical $\{10,4,2\}$-design on $\mathbb{S}^3$ and a $(4,N,1/2)$ spherical code.

{For $d\ge3$}, let $Q_{\ell} (x) := Q_{d, \ell} (x) = \frac{d + 2 \ell - 2}{d - 2} C_{\ell}^{((d-2)/2)}(x)$ 
be the (scaled) Gegenbauer polynomial of degree $\ell$ in one variable $x$ as introduced in \cite[Definition 2.1]{DelsarteGoethalsSeidel77} {(later we only consider the case $d=4$)}.
{It is also defined by the recurrence relation
\[ \lambda_{\ell+1}Q_{\ell+1}(x)=xQ_\ell(x)-(1-\lambda_{\ell-1})Q_{\ell-1}(x)\]
with the initial values $Q_0(x)=1,\ Q_1(x)=dx$, where $\lambda_\ell=\ell/(d+2\ell -2)$.
It holds that $\dim {\rm Harm}_\ell(\R^d)=Q_\ell(1)$.
The Gegenbauer polynomials $Q_\ell(x)$ are the orthogonal polynomials on the closed interval $[-1, 1]$ with respect to the inner product of the weight function $(1 - x^2)^{(d -3)/2}$, i.e.,
\[
\int_{-1}^{1} Q_{k}(x) Q_{\ell} (x) (1 - x^2)^{(d-3)/2} \; dx = b_{ \ell}
\delta_{k, \ell}
\]
where $b_{\ell}$ is some (normalization) constant depending on $d$ and $\ell$, and $\delta_{k,\ell}$ is the Kronecker delta.
To each real polynomial $F$ of degree $r$ we can associate its \textit{Gegenbauer expansion}
\begin{equation}
\label{eq:F}
F(x) = \sum_{\ell = 0}^{r} f_\ell Q_{\ell} (x),
\end{equation}
where the \textit{Gegenbauer coefficients} $f_\ell$ can be calculated as follows:
\[
f_\ell = \frac{1}{b_{\ell}} \int_{-1}^{1} F(x) Q_{ \ell} (x) (1 - x^2)^{(d - 3)/2} \; dx.
\]}

Let $\{\varphi_{\ell,i}\}_{i=1}^{N_\ell}$ be an orthonormal basis of ${\rm Harm}_{\ell} (\mathbb{S}^{d-1})$ which is the restriction of ${\rm Harm}_\ell(\R^d)$ to $\mathbb{S}^{d-1}$, 
where $N_{\ell} :=N_{d,\ell}= \dim {\rm Harm}_{\ell} (\mathbb{S}^{d-1}) = Q_\ell(1)$.
For a finite subset $X$ of $\mathbb{S}^{d-1}$, we write %((a_{ij)) is standard ??
\[ H_{\ell} :=H_{\ell}(X) = \big( \varphi_{\ell,i}(\xi) \big)_{\begin{subarray}{c} \xi\in X\\ 1\le i\le N_\ell\end{subarray}}\]
for the $|X|\times N_\ell$ matrix whose rows and columns are indexed by $\xi\in X$ and $1\le i\le N_\ell$, respectively.
$H_0$ is of size $|X|\times 1$ whose entries are all 1.
For $\ell\ge1$, one has ${}^tH_{\ell} H_0 = \big( \sum_{\xi\in X}\varphi_{\ell,i}(\xi)\big)_{1\le i\le N_{\ell}}$.
From this, we see that $X$ is a spherical $T$-design if and only if $\|{}^tH_\ell H_0\|=0$ holds for all $\ell \in T$, where for a real matrix $M=(a_{ij})$, we write $\| M\| :=\sum_{i,j} a_{ij}^2$.

A key lemma for the linear programming method is as follows (cf.~\cite[Corollary 3.8]{DelsarteGoethalsSeidel77}).

\begin{lem}\label{lem:fisher}
Let $X\subset\mathbb{S}^{d-1}$ be a finite subset.
For a real polynomial $F(x)\in \R[x]$ with the Gegenbauer expansion \eqref{eq:F}, we have
%\begin{equation}\label{eq:fisher_lemma} 
\[f_0 |X|^2+\sum_{\ell=1}^rf_\ell \| {}^t H_{\ell}H_0\| = F(1) |X| + \sum_{\alpha\in A(X)} F(\alpha)d_\alpha.\]
%\end{equation}
where $d_\alpha:=\sharp \{ (\xi,\eta)\in X\times X\mid \langle \xi,\eta\rangle =\alpha\}$.
\end{lem}
\begin{proof}
{We use the additive formula given in Theorem 3.3 of \cite{DelsarteGoethalsSeidel77}.}
For any $\xi,\eta\in \mathbb{S}^{d-1}$ we have
\begin{equation*}\label{eq:additive} 
\sum_{i=1}^{N_{\ell}} \varphi_{\ell,i}(\xi) \varphi_{\ell,i}(\eta) = Q_\ell(\langle\xi,\eta\rangle).
\end{equation*}
Using this, one computes
\begin{align*}
 \| {}^t H_\ell H_0\| &= \sum_{1\le i\le N_{\ell}}\bigg( \sum_{\xi\in X}\varphi_{\ell,i}(\xi)\bigg)^2=\sum_{\xi,\eta\in X} Q_{\ell} (\langle \xi,\eta\rangle)=\sum_{\alpha\in A(X)\cup\{1\}} Q_{\ell}(\alpha)d_\alpha.
\end{align*}
By linearity, it holds that
\[ \sum_{\ell=0}^rf_\ell \| {}^t H_\ell H_0 \| = \sum_{\alpha\in A(X)\cup\{1\}} F(\alpha)d_\alpha.\]
Now the desired result follows from $ \| {}^t H_0 H_0\| =|X|^2$ and $d_1=|X|$.
\end{proof}

{Now we use Lemma \ref{lem:fisher} to obtain the lower bound for} the cardinality of a spherical $\{10,4,2\}$-design on $\mathbb{S}^3$.
{Hereafter, we set $d$ to be 4.}

\begin{thm}\label{thm:upper_bound_design}
Let $X$ be a spherical $\{10,4,2\}$-design on $\mathbb{S}^3$.
Then, {we have that} $|X|\ge 12$.
Moreover, $X$ attains the lower bound if and only if $X$ is a $(4,12,1/2)$ spherical code with $A(X)\subset \{-1/2,0,1/2\}$.
\end{thm}
\begin{proof}
Consider the real polynomial
\begin{equation}\label{eq:F_T}
\begin{aligned}
F_T(x) &:= \frac{1}{11264} Q_{10}(x) + \frac{1}{2560} Q_{4} (x) + 
\frac{1}{768} Q_{2} (x) + \frac{3}{1024} \\
&= \frac{1}{16} x^2 (x + \frac{1}{2} )^2  (x - \frac{1}{2} )^2 (16 x^4 - 28 x^2 + 13).
\end{aligned}
\end{equation}
We write $F_T(x)=\sum_{\ell=0}^{10} f_\ell Q_\ell(x)$.
One can easily check the inequality $F_T(x)\ge0$ for all $x\in [-1,1)$.
{Since $\| {}^t H_\ell H_0 \|=0$ for $\ell\in T$}, by Lemma \ref{lem:fisher}, we get the inequality
\begin{equation}\label{eq:ineq_F_T}
f_0|X|^2-F_T(1)|X| = \sum_{\alpha\in A(X)} F_T(\alpha)d_\alpha\ge0.
\end{equation}
Since $F_T(1)=\frac{9}{256}$, the desired inequality $|X| \geq F_T(1)/f_0=12$ follows.
The equality holds if $F_T(\alpha)=0\ (\forall \alpha\in A(X))$.
We complete the proof, because $\{\alpha \in \R \mid F_T(\alpha)=0\} = \{-1/2,0,1/2\}$.
\end{proof}

%\begin{cor}\label{cor:equivalence_code_design}
%Let $X$ be a subset of $\mathbb{S}^3$ with $12$ points.
%The following two statements are equivalent.
%\begin{itemize}
%\item[(1)] $X$ is a antipodal spherical $\{10,4,2\}$-design.
%\item[(2)] $X$ is a half set of an antipodal $(4,24,1/2)$ spherical code with $A(X)\subset\{-1/2,0,1/2\}$.
%\end{itemize}
%\end{cor}
%\begin{proof}
%Assume (1). 
%Then $X$ is a spherical $\{10,4,2\}$-design with 12 points.
%Hence (2) is obtained as a consequence of Theorem \ref{thm:upper_bound_design}.
%To deduce (1) from (2), we notice that $X$ is an (4,12,1/2) spherical code with $A(X)\subset [-1/2,1/2]$.
%Then the result follows from Theorem \ref{thm:lower_bound_code}.
%\end{proof}

An antipodal spherical $\{10,4,2\}$-design $X\subset \mathbb{S}^3$ is said to be {\itshape tight} when $|X|=24$.
From Lemma \ref{lem:half set of antipodal T-design} and Theorem \ref{thm:upper_bound_design}, the existence of a tight antipodal spherical $\{10,4,2\}$-design $X\subset\mathbb{S}^3$ is equivalent to that of a spherical $\{10,4,2\}$-design $Y\subset\mathbb{S}^3$ with 12 points as the correspondence $X=Y\cup (-Y)$. 
%In this sense, we may consider a spherical $\{10,4,2\}$-design on $\mathbb{S}^3$ with 12 points to be tight (and this definition is compatible with the one given in \cite[Definition 6.4]{BBTZ17}), but we only deal with antipodal sets.
%Note that our tight design means a `minimal' design  
%proved by the linear programming bound obtained from the test function \eqref{eq:F_T}, and 
Our tight design means a `minimal' antipodal design proved by the linear programming bound obtained from the test function \eqref{eq:F_T}, and
it is different from the classical definition of tight spherical $t$-designs given in Introduction. 
Other definitions for tight spherical $T$-designs can be found in \cite[Definition 6.4]{BBTZ17} and \cite{BannaiOkudaTagami15}, where the existence and non-existence of tight spherical $T$-designs are studied.
Several investigations have been conducted in this direction; see e.g., \cite{OkudaYu16,ZhuBannaiBannaiKimYu17}.

Theorem \ref{thm:upper_bound_design} says that every half set of a tight antipodal spherical $\{10,4,2\}$-design on $\mathbb{S}^3$ is a $(4,12,1/2)$ spherical code. 
The natural question to ask is the upper bound of $N$ for a $(4,N,1/2)$ spherical code.

\begin{thm}\label{thm:lower_bound_code}
Let $X$ be a $(4,N,1/2)$ spherical code with $A(X)\subset [-1/2,1/2]$.
Then we have that $N\le12$.
Furthermore, $X$ attains the upper bound if and only if $X$ is a spherical $\{10,4,2\}$-design and $A(X) \subset\{-1/2,0,1/2\}$. 
\end{thm}
\begin{proof}
For 
% $a_1 < -1$, 
$a_1 \ge0$, 
let us consider the function
\begin{equation}\label{eq:F_C}
\begin{aligned}
F_C(x) &:= 
\frac{1}{11264} Q_{10}(x) + \frac{64 a_1 + 15}{5120} Q_{4}(x) 
+  \frac{64 a_1 + 15}{1536} Q_{2}(x) + \frac{4 a_1 + 1}{64} \\
%%%
&= x^2 \left (x + \frac{1}{2} \right ) \left (x - \frac{1}{2} \right ) \left (x^6 - 2 x^4  + \frac{5}{4} x^2 +  a_1 \right ).
\end{aligned}
\end{equation}
The inequality $x^6 - 2 x^4  + \frac{5}{4} x^2\ge0$ holds for any $x\in [-1/2,1/2]$ which implies 
\[F_C(\alpha)\le0 \quad (\forall \alpha\in [-1/2,1/2]).\]
From Lemma \ref{lem:fisher} and the assumption $A(X)\subset [-1/2,1/2]$, we get the inequality
\begin{equation}\label{eq:ineq_F_C}
F_C(1)|X| -f_0|X|^2 = -\sum_{\alpha\in A(X)} F_C(\alpha)d_\alpha+\sum_{\substack{\ell=1}}^{10} f_\ell \| {}^tH_\ell H_0 \| \ge0,
\end{equation}
where $f_\ell$ denotes the coefficient of $F_C$ {corresponding to} $Q_\ell$.
Since $F_C(1)=\frac{3(4a_1+1)}{16}>0$, we obtain $F_C(1)/f_0=12\ge |X|=N$.
The equality in \eqref{eq:ineq_F_C} holds if and only if $F_C(\alpha)=0 \ (\forall \alpha \in A(X))$ and $\| {}^t H_{\ell} H_0 \|=0$ for all $\ell\in \{10,4,2\}$.
The desired result then follows from $\{\alpha \in \R \mid F_C(\alpha)=0\} = \{-1/2,0,1/2\}$.
\end{proof}

%%%%%%%%%%%%%%%%%%%%%%%%%%%%%%%%%%
\section{The $D_4$ lattice and spherical $\{10,4,2\}$-designs}
This section gives the construction of a tight antipodal spherical $\{10,4,2\}$-design on $\mathbb{S}^3$ from the shells of the $D_4$ lattice.

Following \cite[Section 1.4]{Ebeling}, we define {\itshape the $D_4$ lattice} by 
\[D_4:=\{\boldsymbol{x}=(x_1,x_2,x_3,x_4)\in \Z^4 \mid x_1+x_2+x_3+x_4\equiv 0 \mod 2\}.\]
For $m\in \Z_{\ge0}$, {\itshape the $m$-shell} of the $D_4$ lattice is denoted by
\[ \big(D_4\big)_m := \{ \boldsymbol{x} \in D_4 \mid x_1^2+x_2^2+x_3^2+x_4^2=m\}.\]
It follows that $\big(D_4\big)_m= \varnothing$, if $m$ is odd.
When $m$ is even, $ \big(D_4\big)_m$ is not the empty set because of the Jacobi's four-square theorem {(see e.g., \cite[p.19]{DS})}
\begin{equation}\label{eq:Jacobi}
|(D_4)_{2m}|=24\sum_{\substack{d\mid 2m\\ d:{\rm odd}}}d.
\end{equation}
For instance, the $2$-shell $(D_4)_{2}$ (the set of minimal vectors of $D_4$) consists of 24 points; all permutations of $(\pm1,\pm1,0,0)$.  
Note that the $2$-shell $(D_4)_{2}$, which is called {\itshape the $D_4$ root system}, generates the $D_4$ lattice.
{We set $\mathbf{D}_4:=(D_4)_2$}.

We now prove that the normalized set 
\[ \frac{1}{\sqrt{2m}}\big(D_4\big)_{2m}:=\left\{ \frac{1}{\sqrt{2m}}\boldsymbol{x}\ \middle| \ \boldsymbol{x}\in \big(D_4\big)_{2m}\right\}\]
on $\mathbb{S}^3$ is an example of antipodal spherical $\{10,4,2\}$-designs.
There are at least two proofs of this.
One is based on some spherical design properties on group orbits.
The other uses the theory of modular forms, which will be mentioned in Remark \ref{rem:another proof of the strength}.
Here we give the former proof.

We recall that the orthogonal transformation group 
\[O(\R^d):=\{\sigma:\R^d\rightarrow\R^d\mid \langle \sigma(\boldsymbol{x}),\sigma(\boldsymbol{y})\rangle=\langle \boldsymbol{x},\boldsymbol{y}\rangle\ \mbox{for all}\  \boldsymbol{x},\boldsymbol{y}\in \R^d\}\]
of $\R^d$ acts on ${\rm Harm}_\ell(\R^{d})$ by $\big(\sigma^\ast P \big)(\boldsymbol{x}):=P(\sigma(\boldsymbol{x}))$ for $P\in {\rm Harm}_\ell(\R^{d})$ and $\sigma\in O(\R^d)$.
For a subgroup $G$ of $O(\R^d)$, the $G$-invariant subspace of ${\rm Harm}_{\ell} (\R^d)$ is denoted by ${\rm Harm}_\ell (\R^d)^G :=\{ P\in {\rm Harm}_\ell (\R^d) \mid \sigma^\ast P=P\ \mbox{for all}\ \sigma \in G\}$. 

\begin{lem}\label{lem:inv_harm}
For any finite subgroup $G$ of $O(\R^d)$ and $\boldsymbol{x}\in \mathbb{S}^{d-1}$, the $G$-orbit $\boldsymbol{x}^G:=\{\sigma(\boldsymbol{x})\in \mathbb{S}^{d-1}\mid \sigma\in G\} $ is a spherical $T$-design with $T=\{\ell \in \N \mid \dim {\rm Harm}_\ell (\R^d)^G=0 \}$.
Moreover, if $G$ has $-I$, which sends $\boldsymbol{y}$ to $-\boldsymbol{y}$ for $\boldsymbol{y}\in \R^d$, then $\boldsymbol{x}^G$ is antipodal, and its {every} half set is a spherical $T'$-design with $T'=\{2\ell \in 2\N \mid \dim {\rm Harm}_{2\ell} (\R^d)^G=0 \}$. 
\end{lem}
\begin{proof}
Let $G_{\boldsymbol{x}}$ denote the stabilizer subgroup of $\boldsymbol{x}$.
For $P\in {\rm Harm}_{\ell}(\R^d)$, we have
\[ \sum_{\boldsymbol{y}\in \boldsymbol{x}^G} P(\boldsymbol{y})= \frac{1}{|G_{\boldsymbol{x}}|}\sum_{\sigma \in G} (\sigma^\ast P)(\boldsymbol{x}). \]
The first statement follows from the fact that the map $ {\rm Harm}_{\ell} (\R^d)\rightarrow  {\rm Harm}_{\ell} (\R^d)^G,P\mapsto \sum_{\sigma \in G} (\sigma^\ast P)$ is surjective.

Suppose that $-I\in G$.
We have $-\boldsymbol{y}\in \boldsymbol{x}^G$ for any $\boldsymbol{y}\in \boldsymbol{x}^G$.
Hence $\boldsymbol{x}^G$ is antipodal.
The latter statement follows from Lemma \ref{lem:half set of antipodal T-design}.
%Now let $X'$ be a half set of $\boldsymbol{x}^G$.
%For any even polynomial $P$, we have
%\[\sum_{\boldsymbol{y}\in \boldsymbol{x}^G} P(\boldsymbol{y})=2\sum_{\boldsymbol{y}\in X'} P(\boldsymbol{y}),\]
%from which the latter statement follows.
\end{proof}

We note that for spherical $T$-designs $X_1$ and $X_2$ on $\mathbb{S}^{d-1}$, the union $X_1\cup X_2$ is {also} a spherical $T$-design if $X_1\cap X_2 =\varnothing$.

\begin{prop}\label{prop:D4 root}
For any $m\ge1$, the subset $\frac{1}{\sqrt{2m}}\big(D_4\big)_{2m}$ of $\mathbb{S}^3$ is an antipodal spherical $\{10,4,2\}$-design.
{Moreover}, for any $n\ge1$ the set {$\frac{1}{\sqrt{2^{n}}}\big(D_4\big)_{2^n}$} is a tight antipodal spherical $\{10,4,2\}$-design on $\mathbb{S}^3$.
\end{prop}
\begin{proof}
We use the fact that the root system $\mathbf{D}_4$ is invariant under the action of the Weyl group $W(\mathbf{F}_4)$ of the root system $\mathbf{F}_4$ (this fact is already pointed out in \cite[Proposition {23}]{Pache05}).
The group $W(\mathbf{F}_4)$ is a discrete subgroup of $O(\R^4)$ of order 1152 and coincides with the automorphism group ${\rm Aut}(\mathbf{D}_4):=\{\sigma \in O(\R^4)\mid \sigma(\mathbf{D}_4)=\mathbf{D}_4\}$ of the root system $\mathbf{D}_4$ (see \cite{Bourbaki}).
Since the $D_4$ lattice is generated by the set $\mathbf{D}_4$, the set $\frac{1}{\sqrt{2m}}\big(D_4\big)_{2m}$ is also invariant under the action of $W(\mathbf{F}_4)$, and hence it has a $W(\mathbf{F}_4)$-orbit decomposition.
%This shows that $X=\frac{1}{\sqrt{2}}\Lambda_1$ is a spherical $\{10,4,2\}$-design if $\sum_{\boldsymbol{x}\in X} P(\boldsymbol{x}) =0$ for all $W(\mathbf{F}_4)$-invariant harmonic polynomials $P$ of degree $\ell \in\{10,4,2\}$.
The harmonic Molien series for $W(\mathbf{F}_4)$ is calculated with the exponents $(m_1,m_2,m_3,m_4)=(1,5,7,11)$ (see e.g., \cite[Theorem 4.6]{GoethalsSeidel}), namely,
\begin{equation}\label{eq:har_molien}
\begin{aligned}
\sum_{\ell \ge0} \dim{\rm Harm}_\ell (\R^4)^{W(\mathbf{F}_4)} t^\ell&=(1-t^2)
\prod_{i=1}^4\frac{1}{1-t^{m_i+1}}= \frac{1}{(1-t^6)(1-t^8)(1-t^{12})}\\ &=1+t^6+t^8+2t^{12}+t^{14}+t^{16}+2t^{18}+\cdots.
%2t^{20}+t^{22}+4t^{24}+\cdots.
\end{aligned}
\end{equation}
With this, the result follows from Lemma \ref{lem:inv_harm}.
The `Moreover' part follows from \eqref{eq:Jacobi}, namely, {that we have} $|(D_4)_{2^n}|=24$.
\end{proof}

Combining Proposition \ref{prop:D4 root} with Lemma \ref{lem:half set of antipodal T-design}, we see that every half set of $\frac{1}{\sqrt{2m}}\big(D_4\big)_{2m}$ is a spherical $\{10,4,2\}$-design.
In particular, it follows from Theorem \ref{thm:upper_bound_design} that every half set $X$ of $\frac{1}{\sqrt{2^{n}}}\big(D_4\big)_{2^n}$ is a $(4,12,1/2)$ spherical code with $A(X)\subset \{-1/2,0,1/2\}$.
Indeed, one can check that the inner product set of the normalized $D_4$ root system $\frac{1}{\sqrt{2}}\mathbf{D}_4$ is given by
\[ A\left(\frac{1}{\sqrt{2}}\mathbf{D}_4\right) = \left\{-1,-\frac12,0,\frac12\right\}.\]
%We note that by Corollary \ref{cor:equivalence_code_design}, every half set $X$ of $\frac{1}{\sqrt{2}}\mathbf{D}_4$ is a $(4,12,1/2)$ spherical code with $A(X)\subset \{-1/2,0,1/2\}$.

\begin{rem}
According to \cite[Proposition 2]{BannaiZhaoZhuZhu18}, there exists a half set of $\frac{1}{\sqrt{2}}\mathbf{D}_4$ such that it is a spherical $\{10,4,2,1\}$-design {(a half set is not antipodal, so this is non-trivial)}.
\end{rem}

%%%%%%%%%%%%%%%%%%%%%%%%%%%%%%%%%%
\section{Uniqueness of the antipodal spherical $\{10,4,2\}$-design}

In this section, we prove Theorem \ref{thm:classification}.
Let $X\subset\mathbb{S}^{d-1}$ be a spherical $T$-design. 
For any orthogonal transformation $\sigma\in O(\R^d)$, the set $\sigma(X)=\{\sigma(\boldsymbol{x})\mid \boldsymbol{x}\in X\}$ is again a spherical $T$-design.
Thus, an orthogonal transformation of $\frac{1}{\sqrt{2}}\mathbf{D}_4$ is still a tight antipodal spherical $\{10,4,2\}$-design on $\mathbb{S}^3$.
The goal is to prove the opposite statement, namely, any antipodal spherical $\{10,4,2\}$-design on $\mathbb{S}^3$ with 24 points is obtained from an orthogonal transformation of $\frac{1}{\sqrt{2}}\mathbf{D}_4$, which can be referred as a uniqueness theorem in the study of the classification of spherical designs.

%In what follows, we prove the uniqueness of the antipodal $(4,24,1/2)$ spherical code (from which the uniqueness of the antipodal spherical $\{10,4,2\}$-design on $\mathbb{S}^3$ with 24 points follows), along the line .

Our proof is along the line of the proof of the uniqueness of the 600-cell $C_{600} \subset\mathbb{S}^3$ as a spherical $11$-design with $120$ points, given by Boyvalenkov-Danev \cite{BoyvalenkovDanev01}.
Let us first recall some relevant materials from it.

For $\boldsymbol{y}\in \mathbb{S}^{d-1}$ and a finite subset $X\subset \mathbb{S}^{d-1}$, we let 
\[A^{\boldsymbol{y}}(X):=\{\alpha\in [-1,1]\mid \mbox{there exists $\boldsymbol{x}\in X$ such that $\langle\boldsymbol{x},\boldsymbol{y}\rangle=\alpha$}\},\]
and for $\alpha \in [-1,1]$, we write $\widetilde{X}_{\alpha}^{\boldsymbol{y}}:=\{\boldsymbol{x}\in X \mid \langle \boldsymbol{x},\boldsymbol{y}\rangle=\alpha\}$.
Note that if $\boldsymbol{y}\in X$, then $1\in A^{\boldsymbol{y}}(X)\subset A(X)\cup\{1\}$.
The sequence of positive integers $(A^{\boldsymbol{y}}_\alpha (X))_{\alpha \in A^{\boldsymbol{y}}(X)}$, {where} $A_\alpha^{\boldsymbol{y}}(X):=|\widetilde{X}_{\alpha}^{\boldsymbol{y}}|$ is called the {\itshape distance distribution} of $X$ with respect to $\boldsymbol{y}$.
When a subset $X\subset \mathbb{S}^{d-1}$ is a spherical $t$-design {such that $|A^{\boldsymbol{y}}(X)|\le t+1$ for some $\boldsymbol{y}\in \mathbb{S}^{d-1}$}, the distance distribution of $X$ with respect to $\boldsymbol{y}$ is obtained as the unique solution to the Vandermonde system
\begin{equation}\label{eq:Vandermonde} 
\sum_{\alpha\in A^{\boldsymbol{y}}(X)} A_{\alpha}^{\boldsymbol{y}}(X) \alpha^j = a_j |X|, \quad j=0,1,\ldots,|A^{\boldsymbol{y}}(X)|-1, 
\end{equation}
where we set $a_0:=1,a_{2j}:=\frac{(2j-1)!!}{d(d+2)\cdots (d+2j-2)}$ and $a_{2j+1}:=0$ for $j\ge 1$ (the proof of \eqref{eq:Vandermonde} can be done by taking $F(x)=x^j,\ j=0,1,\ldots,t$, in the following equivalent definition of a spherical $t$-design \cite[Corollary 3.8, Theorem 5.5]{DelsarteGoethalsSeidel77}; for a finite set $X\subset \mathbb{S}^{d-1}$, $X$ is a spherical $t$-design if and only if for any $\boldsymbol{y}\in \mathbb{S}^{d-1}$ the equality $\sum_{\boldsymbol{x}\in X}F(\langle \boldsymbol{x},\boldsymbol{y}\rangle)=|X|f_0$ holds for all $F(x)\in \R[x]$ of degree at most $t$, where $f_0$ is the constant term of the Gegenbauer expansion of $F$ as in \eqref{eq:F}: See also \cite[Section 2.1]{Boyvalenkov95}).

{Following \cite{DelsarteGoethalsSeidel77}, we say that a set $X\subset \mathbb{S}^{d-1}$ with $N$ points} is called a $(d,N,s,t)$ {\itshape configuration}, if $X$ is a spherical $t$-design such that $s=|A(X)|$.
It follows that for $\boldsymbol{y}\in X$ and a $(d,N,s,t)$ configuration $X$ with $s\le t+1$, the Vandermonde system \eqref{eq:Vandermonde} (note that $1\in A^{\boldsymbol{y}}(X)$ and $A_1^{\boldsymbol{y}}(X)=1$) has the unique solution, because {the number $|A^{\boldsymbol{y}}(X)|-1$ of variables in the system of linear equations \eqref{eq:Vandermonde} is less than or equal to $t+1$}.
In this case, $A_{\alpha}^{\boldsymbol{y}}(X)$ does not depend on the choice of $\boldsymbol{y}\in X$ and we write $A_{\alpha}(X):=A_{\alpha}^{\boldsymbol{y}}(X)$.

\begin{thm}\label{thm:uniqueness}
For any tight antipodal spherical $\{10,4,2\}$-design $X$ on $\mathbb{S}^3$, there exists an orthogonal transformation $\sigma \in O(\R^4)$ such that $X=\sigma\big( \frac{1}{\sqrt{2}}\mathbf{D}_4\big) $.
\end{thm}

\begin{proof}
By Lemma \ref{lem:half set of antipodal T-design} and Theorem \ref{thm:upper_bound_design}, a half set $X'$ of $X$ is a $(4,12,1/2)$ spherical code with $A(X')\subset \{-1/2,0,1/2\}$, so $A(X)\subset \{-1,-1/2,0,1/2\}$.
Since $X$ is a $(4,24,s,5)$ configuration with $s\le 4$, the distance distribution $(A^{\boldsymbol{y}}_\alpha (X))_{\alpha \in A^{\boldsymbol{y}}(X)}$ of $X$ does not depend on the choice of $\boldsymbol{y}\in X$.
Solving the equations \eqref{eq:Vandermonde}, we get
\[ A_{-1}(X)=1, \quad A_{-\frac12}(X)=A_{\frac12}(X)=8, \quad A_{0}(X)=6,\]
which implies that $X$ is a $(4,24,4,5)$ configuration.

For each $\alpha \in A(X)\setminus\{-1\}$, we now recall a derived code $X_{\alpha}\subset \mathbb{S}^2$ of $X$ introduced in \cite[Section 8]{DelsarteGoethalsSeidel77}.
We may assume $\boldsymbol{e}=(0,0,0,1)\in X$ (if not, one can take $\tau \in O(\R^4)$ such that $\boldsymbol{e}\in \tau(X)$).
For any $\boldsymbol{x}\in \widetilde{X}_{\alpha}^{\boldsymbol{e}}$, it holds that
\[\frac{1}{\sqrt{1-\alpha^2}}(\boldsymbol{x}-\alpha\boldsymbol{e})\in \{\boldsymbol{y}\in \mathbb{S}^3\mid \langle \boldsymbol{y},\boldsymbol{e}\rangle =0\}=\{(y_1,y_2,y_3,0)\in \mathbb{S}^3\}.\]
Thus, the image of $ \widetilde{X}_{\alpha}^{\boldsymbol{e}}\subset\mathbb{S}^3$ under the composition map
\[
\begin{array}{cccccccc}
p_\alpha :& \R^4 & {\longrightarrow} & \R^4 & {\longrightarrow} & \R^3,\\
& \boldsymbol{x} & \longmapsto & \frac{1}{\sqrt{1-\alpha^2}}(\boldsymbol{x}-\alpha\boldsymbol{e})=(z_1,z_2,z_3,z_4) & \longmapsto & (z_1,z_2,z_3).
\end{array}
\]
lies in $\mathbb{S}^2$.
The image 
\[X_{\alpha}:= p_{\alpha}\big(\widetilde{X}_{\alpha}^{\boldsymbol{e}}\big)\subset \mathbb{S}^2,\]
called the {\itshape derived code}, is also a spherical design with the strength weakened (see \cite[Theorem 8.2]{DelsarteGoethalsSeidel77} for more details).
In our case, $X_{\alpha}$ becomes a spherical $3$-design on $\mathbb{S}^2$.

Let us consider the inner product set $A(X_\alpha)$ for each $\alpha\in\{0,\pm\frac12\}$.
By definition, one easily finds that $A(X_{\alpha}) \subset \left\{ \frac{\beta-\alpha^2}{1-\alpha^2} \, \middle| \, \beta\in A(X)\right\}$.
Computing the terms $\frac{\beta-\alpha^2}{1-\alpha^2} $, we get
\[ A(X_{\pm\frac12}) \subset \left\{-1,-\frac13,\frac13\right\} \quad \mbox{and}\quad A(X_{0}) \subset \left\{-1,-\frac12,0,\frac12\right\}.\]
Namely, the sets $X_{\pm\frac12}$ and $X_0$ are $(3,8,s_1,3)$ and $(3,6,s_2,3)$ configurations with $s_1\le 3$ and $s_2\le 4$, respectively.
For each $X_\alpha$, one can compute the unique solution to the Vandermonde system \eqref{eq:Vandermonde}.
Indeed, we have that
\begin{align*}
&A_{-1}\big(X_{\pm\frac12}\big) =1,\quad A_{\pm\frac13}\big(X_{\pm\frac12} \big)=3,\\
&\quad A_{-1}\big(X_{0} \big)=1,\quad A_{\pm\frac12}\big(X_{0} \big)=0,\quad A_{0}\big(X_{0} \big)=4.
\end{align*}
Hence the sets $X_{\pm\frac12}$ and $X_0$ are $(3,8,3,3)$ and $(3,6,2,3)$ configurations, respectively.
Both $A_{-1}(X_\alpha)=1$ and its independence of the choice of $\boldsymbol{y}\in X_{\alpha}$ imply $X_\alpha$ being antipodal.
Remark that the antipodal $(3,6,2,3)$ configuration $X_0$, which by \cite[Theorem 6.8]{DelsarteGoethalsSeidel77} is a tight antipodal spherical $3$-design on $\mathbb{S}^2$, is an orthogonal transformation of the set $C_6:=\{(\pm1,0,0),(0,\pm1,0),(0,0,\pm1)\}$ of vertices of the regular octahedron.
This shows that there exists an orthogonal transformation $\sigma'\in O(\R^3)$ such that 
\[X_0=\sigma'(C_6).\]

We now prove that $X_{-\frac12}=X_{\frac12}$ and $X_{\frac12}=\sigma'(C_8)$, where $C_8:=\big\{ \big( \pm\frac{1}{\sqrt{3}}, \pm\frac{1}{\sqrt{3}} , \pm\frac{1}{\sqrt{3}} \big)\big\}$.
It can be checked that the distance distribution of $X_{-\frac12}$ with respect to $\boldsymbol{y}\in X_{\frac12}$ satisfies
\[ A^{\boldsymbol{y}}\big(X_{-\frac12}\big)\subset \left\{ -1,-\frac{1}{3},\frac13,1\right\},\]
because, by definition of the derived code, $\langle \boldsymbol{x},\boldsymbol{y}\rangle \in \left\{ \frac{\alpha+\frac14}{1-\frac{1}{4}}\, \middle| \, \alpha \in A(X) \right\} $ holds for all $\boldsymbol{x}\in X_{-\frac12}$.
Thus, $|A^{\boldsymbol{y}}\big(X_{-\frac12}\big)|\le 4$, and hence, one can solve the Vandermonde system \eqref{eq:Vandermonde} to get
\[A_{-1}^{\boldsymbol{y}}\big(X_{-\frac12}\big)=1,\quad A_{-\frac13}^{\boldsymbol{y}}\big(X_{-\frac12}\big)=3,\quad A_{\frac13}^{\boldsymbol{y}}\big(X_{-\frac12}\big)=3,\quad A_{1}^{\boldsymbol{y}}\big(X_{-\frac12}\big)=1.\]
The last equality implies $\boldsymbol{y}\in X_{-\frac12}$.
Since the above equation holds for any $\boldsymbol{y}\in X_{\frac12}$, one finds that $X_{\frac12} \subset X_{-\frac12}$, which implies the desired equality $X_{\frac12} = X_{-\frac12}$.
To show that $C_8=(\sigma')^{-1}\big( X_{\frac12}\big)$, we again compute the distance distribution of $C_6$ with respect to $\boldsymbol{y}\in C_8$ by the Vandermonde system \eqref{eq:Vandermonde} and it holds that
\[A^{\boldsymbol{y}}\big(C_6\big)=\left\{ -\frac{1}{\sqrt{3}}, \frac{1}{\sqrt{3}}\right\}\quad\mbox{and}\quad A^{\boldsymbol{y}}_{\pm\frac{1}{\sqrt{3}}} \big(C_6\big)  =3.\]
Namely, $\boldsymbol{y}\in C_8$ satisfies $\langle \boldsymbol{x},\boldsymbol{y}\rangle =\pm\frac{1}{\sqrt{3}}$ for all $\boldsymbol{x}\in C_6$.
This implies that $C_8\subset \big\{ \big( \pm\frac{1}{\sqrt{3}}, \pm\frac{1}{\sqrt{3}} , \pm\frac{1}{\sqrt{3}} \big)\big\}$.
Since $|C_8|=8$, the equality holds.

%We now check that $X_{\pm\frac12}$ is (up to orthogonal transformations) uniquely determined from $X_0$.
%For this, take $\sigma'\in O(\R^3)$ such that $X_0=\sigma'(C_6)$ and set $C_8=(\sigma')^{-1}\big( X_{\frac12}\big)$.
%We prove $C_8=\big\{ \big( \pm\frac{1}{\sqrt{3}}, \pm\frac{1}{\sqrt{3}} , \pm\frac{1}{\sqrt{3}} \big)\big\}$.
%Again one can compute the distance distribution of $C_6$ with respect to $\boldsymbol{y}\in C_8$ by the Vandermonde system \eqref{eq:Vandermonde} and it holds that
%\[A^{\boldsymbol{y}}\big(C_6\big)=\left\{ -\frac{1}{\sqrt{3}}, \frac{1}{\sqrt{3}}\right\}\quad\mbox{and}\quad A^{\boldsymbol{y}}_{\pm\frac{1}{\sqrt{3}}} \big(C_6\big)  =3.\]
%Namely, $\boldsymbol{y}\in C_8$ satisfies $\langle \boldsymbol{x},\boldsymbol{y}\rangle =\pm\frac{1}{\sqrt{3}}$ for all $\boldsymbol{x}\in C_6$.
%This shows that $C_8\subset \big\{ \big( \pm\frac{1}{\sqrt{3}}, \pm\frac{1}{\sqrt{3}} , \pm\frac{1}{\sqrt{3}} \big)\big\}$.
%Since $|C_8|=8$, the equality holds.

%As a result, the whole set $X$ is constructed with the help of only $C_6$, which is unique up to orthogonal transformations.
%Hence, we complete the proof.

Finally, we prove that $X=\sigma\big( \frac{1}{\sqrt{2}}\mathbf{D}_4\big) $ with $\sigma=(\sigma'\otimes 1)\circ R\in O(\R^4)$, where we set $(\sigma'\otimes 1)(x_1,x_2,x_3,x_4)=(\sigma'(x_1,x_2,x_3),x_4)$ and $R$ is a rotation whose matrix representation with respect to the standard basis is given by
\[  \begin{pmatrix}\cos \frac{\pi}{4}&-\sin \frac{\pi}{4}&0&0\\ \sin \frac{\pi}{4}&\cos \frac{\pi}{4}&0&0 \\ 0&0&\cos \frac{\pi}{4}&-\sin \frac{\pi}{4}\\ 0&0&\sin \frac{\pi}{4}&\cos \frac{\pi}{4}\end{pmatrix}.\]
Define the map $q_\alpha: \mathbb{S}^2\rightarrow \mathbb{S}^3$ that sends $(x_1,x_2,x_3)$ to $\sqrt{1-\alpha^2}(x_1,x_2,x_3,\frac{\alpha}{\sqrt{1-\alpha^2}})$.
Note that $q_\alpha(X_\alpha) = \widetilde{X}_{\alpha}^{\boldsymbol{e}}$.
For simplicity, for $X'\subset \mathbb{S}^2$ and $-1<\alpha<1$, we set $(X',\alpha):=\{(x_1,x_2,x_3,\alpha)\mid (x_1,x_2,x_3)\in (\sqrt{1-\alpha^2}) X'\}$, which is $q_\alpha(X')$.
With this, one computes
\begin{align*}
(\sigma'\otimes 1)\big(q_0(C_6)\big) = (\sigma'(C_6),0)=(X_0,0)=\widetilde{X}_0^{\boldsymbol{e}}
\end{align*}
and
\begin{align*}
(\sigma'\otimes 1)\big(q_{\pm\frac12}(C_8)\big) = \left(\sigma'(C_8),\pm\frac{1}{2}\right)=\left(X_{\pm\frac12},\pm\frac{1}{2}\right)=\widetilde{X}_{\pm \frac12}^{\boldsymbol{e}}.
\end{align*}
Since
\[ \{\pm \boldsymbol{e}\}\cup q_0(C_6)\cup q_{\frac12}(C_{8})\cup q_{-\frac12}(C_{8}) = R\left(\frac{1}{\sqrt{2}}\mathbf{D}_4\right)  \]
and $X=\{\pm \boldsymbol{e}\}\cup \widetilde{X}_0^{\boldsymbol{e}}\cup \widetilde{X}_{\frac12}^{\boldsymbol{e}}\cup \widetilde{X}_{-\frac12}^{\boldsymbol{e}}$, we get
\[  (\sigma'\otimes 1)\circ R\left(\frac{1}{\sqrt{2}}\mathbf{D}_4\right) = X.\]
This completes the proof.
\end{proof}

For comparison, we mention the other combinatorial structures on the $D_4$ root system $\mathbf{D}_4$ without going into details. 
%The set of inner products between two distinct points in $\mathbf{D}_4$ is 
%$\{-2,-1,0,1\}$, whose size is $s=4$. 
%The strength as spherical design of $\mathbf{D}_4$ is $t=5$. 
%Since $\mathbf{D}_4$ satisfies the inequality $t\geq 2s-3$, 
The set $\mathbf{D}_4$ has the structure of a $Q$-polynomial association scheme \cite{BB09} (this is verified because the inequality $t\geq 2s-3$ holds for $\mathbf{D}_4$, where $s$ is the size of the set of inner products between two distinct points and $t$ is the strength). 
The set $\mathbf{D}_4$ also has the structure of a kissing number configuration on $\mathbb{S}^3$ \cite{BV08, M08}. 
%The linear programming method due to Delsarte et al \cite{DelsarteGoethalsSeidel77} is not applicable for a proof of this kissing number configuration, but 
The positive semidefinite programming method is directly applicable for a proof of this kissing number \cite{BV08}. 
On the other hand, the set $\mathbf{D}_4$ is not universally optimal code \cite{CCEK07}. 
Any set satisfying $t \geq 2s-1$ is universally optimal, so the strength of $\mathbf{D}_4$ is not strong enough to give the optimality by itself. 
Compared to these results, our main result provides a new characterization of $\mathbf{D}_4$ for the design aspect.

\begin{rem}\label{rem:unique}
We briefly mention some of known uniqueness results.
Each of the 600-cell $C_{600}\subset \mathbb{S}^3$ \cite{BoyvalenkovDanev01}, the normalized $E_8$ root system $\frac{1}{\sqrt{2}}\mathbf{E}_8\subset \mathbb{S}^7$ \cite{BannaiSloane81} and the set of minimal vectors of the Leech lattice $\frac12 \Lambda_{24}\subset \mathbb{S}^{23}$ \cite{BannaiSloane81} is known to be unique as a spherical $t$-design with $N$ poitns, where $t$ and $N$ are indicated as follows.
\[\begin{array}{|c|c|c|c|}  \hline 
X& N & t & T \\ \hline \hline
C_{600} & 120 & 11 & \{58,46,38,34,28,26,22,18,16,14,10,8,6,4,2\} \\ \hline
\frac{1}{\sqrt{2}}\mathbf{E}_8 & 240 & 7 & \{10,6,4,2\} \\ \hline
\frac12 \Lambda_{24} & 196560 & 11 & \{14,10,8,6,4,2\}  \\ \hline
\end{array}\]
They are also unique as an antipodal spherical $T$-design with $N$ points for the above $T\subset \mathbb{N}$ and $N$.
In contrast, our case, the $D_4$ root system, is not unique as a spherical $5$-design (which is a consequence of the result from \cite{CCEK07}) and is unique as an antipodal spherical $\{10,4,2\}$-design.
Namely, the normalized $D_4$ root system is the first example such that it is not unique as a spherical $t$-design, but unique as an antipodal spherical $T$-design.
\end{rem}

%%%%%%%%%%%%%%%%%%%%%%%%%%%%%%%%%%
\section{Application: orthogonal decompositions of shells}

As an application of the uniqueness of the antipodal spherical $\{10,4,2\}$-design on $\mathbb{S}^3$ with 24 points, we now prove that every normalized shell of the $D_4$ lattice is a disjoint union of certain orthogonal transformations of the normalized $D_4$ root system $\frac{1}{\sqrt{2}}\mathbf{D}_4$.

\begin{thm}\label{thm:D_4-decomp}
For any $m\ge1$, there exists a finite subset $S_m \subset O(\R^4)$ such that 
\[ \frac{1}{\sqrt{2m}}\big( D_4\big)_{2m} = \bigsqcup_{\sigma \in S_m} \sigma\big(\frac{1}{\sqrt{2}}\mathbf{D}_4\big). \]
\end{thm}
\begin{proof}
Since the Weyl group $W(\mathbf{F}_4)$ acts on each shells of the $D_4$ lattice, we have a $W(\mathbf{F}_4)$-orbit decomposition of $(D_4)_{2m}$.
Thus, it suffices to show that each orbit $\boldsymbol{x}^{W(\mathbf{F}_4)}$ of $\boldsymbol{x}\in \frac{1}{\sqrt{2m}}\big( D_4\big)_{2m}$ is a disjoint union of certain orthogonal transformations of $\frac{1}{\sqrt{2}}\mathbf{D}_4$.
For this, using Magma system \cite{BCP97}, one can check that there exists a subgroup $N$ of $W(\mathbf{F}_4)$ such that 
\begin{itemize}
  \setlength{\parskip}{0cm} % 段落間
  \setlength{\itemsep}{0cm} % 項目間
    \item $|N|=24${,}
    \item $-I\in N${,}
    \item the harmonic Molien series of $N$ is given by
    \[ \sum_{\ell \ge0} \dim{\rm Harm}_\ell (\R^4)^{N} t^\ell 
    =\sum_{w \in N}\frac{1-t^2}{{\rm Det}(I-t w)}=1+7t^6+9t^8+26t^{12}+\cdots,\]
    where $I$ is the identity matrix. 
\end{itemize}
%Indeed, the subgroup $N$ is generated by $M_1=(m_1m_2m_3m_4)^2$ and $M_2=(m_1m_3m_2m_4)^2$, where 
%$m_i$ are simple reflections of $W(\mathbf{F}_4)$ which are indexed by $m_i m_{i+1}\ne m_{i+1}m_i$. 
%It is known that the Coxeter number of $W(\mathbf{F}_4)$ is $h=12$, and $(m_1m_2m_3m_4)^{h/2}=-I$. 
%Therefore $M_1^3=-I\in W(\mathbf{F}_4)$.
Note that every $W(\mathbf{F}_4)$-orbit has an $N$-orbit decomposition.
It follows from the above data and Lemma \ref{lem:inv_harm} that every half set $X$ of the $N$-orbit $\boldsymbol{x}^{N}$ is a spherical $\{10,4,2\}$-design on $\mathbb{S}^3$ with $|X|\le 12$.
In particular, we see from Theorem \ref{thm:upper_bound_design} that $|X|=12$, and hence that $|\boldsymbol{x}^{N}|=24$.
Thus, by Theorem \ref{thm:uniqueness}, the $N$-orbit $\boldsymbol{x}^{N}$ is an orthogonal transformation of the normalized $D_4$ root system $\frac{1}{\sqrt{2}}\mathbf{D}_4$.
{This completes} the proof.
\end{proof}

\begin{rem}
We briefly mention another proof of Theorem \ref{thm:D_4-decomp}, which provides more information about $S_m$.
It uses Hurwitz quaternions (cf.~\cite{CSb}).
Let $H$ be the ring of Hurwitz quaternions 
\[ H=\left\{ x=x_1 + x_2 i + x_3j +x_4k \ \middle| \ x_1,\ldots,x_4\in \Z \ \mbox{or}\ \Z+\frac12\right\}, \]
where $i^2=j^2=-1$ and $ij=-ji=k$.
This forms a $\Z$-lattice and, for $m\in\Z_{\ge0}$, we obtain the $m$-shell $H_m=\{x\in H\mid x_1^2+x_2^2+x_3^2+x_4^2=m\}$.
A key ingredient is the equality $H_{2m}=\big(D_4\big)_{2m}$ for $m\ge0$ (see \cite[Section 5.5]{CSb}), where we identify $\R+\R i+\R j+\R k$ with $\R^4$ via the isomorphism $x_1 + x_2 i + x_3j +x_4k\mapsto (x_1,x_2,x_3,x_4)$.
From this, we see that a natural choice of the subgroup $N$ of $W(\mathbf{F}_4)$ is the one that is isomorphic to the unit group $H^\times=H_1$, consisting of 24 elements $\pm1,\pm i,\pm j,\pm k,\pm\frac12\pm\frac{i}{2} \pm \frac{j}{2}\pm \frac{k}{2}$, since $H^\times$ acts on $\big(D_4\big)_{2m}$ by right multiplication (this gives rise to an $H^\times$-orbit decomposition of $\big(D_4\big)_{2m}$).
Moreover, the set $S_m$ is taken to be a system of representatives for the right cosets of $H^\times$ in $\big(D_4\big)_{2m}$.
Note that by the Jacobi's four-square theorem \eqref{eq:Jacobi} and $|(D_4)_{2m}|=24|S_m|$, we have 
\begin{equation*}%\label{eq:card S_m}
|S_m|=\sum_{\substack{d\mid 2m\\ d:{\rm odd}}}d .
\end{equation*}
\end{rem}

It might be interesting to ask if there is a similar decomposition of shells of other lattices.

%%%%%%%%%%%%%%%%%%%%%%%%%%%%%%%%%%
\section{Application: the uniqueness of the $D_4$ lattice}

The goal of this section is to give a new proof of the uniqueness of the $D_4$ lattice as an even integral lattice of level 2, which is also another application of the uniqueness of the antipodal spherical $\{10,4,2\}$-design on $\mathbb{S}^3$ with 24 points.
Since the theory of weighted theta functions on a lattice is our key ingredient, we begin with some basic terminologies for lattices and weighted theta functions used in \cite{Ebeling}.

Let $\Lambda\subset \R^d$ be a full-ranked lattice.
The lattice $\Lambda$ is said to be {\itshape integral} (resp. {\itshape even}) if $\Lambda$ is a subset of the dual lattice $\Lambda^\ast: =\{\boldsymbol{y}\in \R^d\mid \langle \boldsymbol{x},\boldsymbol{y}\rangle \in \Z\ \mbox{for all}\ \boldsymbol{x}\in \Lambda\}$ (resp. $\langle \boldsymbol{x},\boldsymbol{x}\rangle \in 2\Z$ for all $\boldsymbol{x}\in \Lambda$).
Let $B$ denote a {$\Z$-basis matrix of $\Lambda$, i.e.} $\Lambda=\{\boldsymbol{m}B\mid \boldsymbol{m}\in \Z^d\}$.
%The determinant ${\rm disc}(\Lambda)=|\det B|$, which does not depend on the choices of a basis matrix, is called the {\itshape discriminant} of the lattice $\Lambda$.
The minimum of all $N\in \N$ with $N \langle \boldsymbol{x},\boldsymbol{x}\rangle \in 2\Z$ for all $\boldsymbol{x}\in \Lambda^\ast$ is called the {\itshape level} of $\Lambda$.

Let $\Lambda$ be an even lattice in $\R^d$ and $\Lambda_{2m}:=\{\boldsymbol{x}\in \Lambda\mid \langle \boldsymbol{x},\boldsymbol{x}\rangle =2m\}$ the $2m$-shell of $\Lambda$.
For $P\in {\rm Harm}_{\ell}(\R^d)$ and $m\ge0$, we write $a_{\Lambda,P}(m):=\sum_{\boldsymbol{x}\in \Lambda_{2m}}P(\boldsymbol{x})$ and define the {\itshape weighted theta function} $\theta_{\Lambda,P}(z)$ by 
\[ \theta_{\Lambda,P}(z) := \sum_{m\ge0} a_{\Lambda,P}(m) q^m \quad (q=e^{2\pi iz} ),\]
which is a holomorphic function on the complex upper half-plane $z\in \mathbb{H}=\{z\in \C \mid {\rm Im} \, z>0\}$.
In particular, if $P=1$ of degree 0, one gets the generating series of the cardinality of each $2m$-shells of $\Lambda$.
Namely, $\theta_{\Lambda,1}(z) = \sum_{m\ge0} | \Lambda_{2m}| q^m$.
 
By Hecke and Schoenberg, for an even integral lattice $\Lambda$ of level $N$ in $\R^d$, the function $\theta_{\Lambda,P}(z)$ is known to be a modular form of weight $d/2+\ell$ {on} $\Gamma_1(N)$ (see e.g., \cite[Chap.3]{Ebeling}), where {$\Gamma_1(N):=\{\gamma\in {\rm SL}_2(\Z) \mid \gamma \equiv (\begin{smallmatrix}1&\ast \\ 0&1\end{smallmatrix})\bmod N\}$ is a congruence subgroup of level $N$ of} ${\rm SL}_2(\Z) $.
Let $M_{k}(\Gamma_1(N))$ denote the $\C$-vector space of modular forms of weight $k$ {on} $\Gamma_1(N)$.
Then we have the $\C$-linear map
\[ \vartheta_{\Lambda,\ell} : {\rm Harm}_{\ell}(\R^d)\otimes_{\R} \C \longrightarrow M_{d/2+\ell}(\Gamma_1(N)) ,\qquad P\longmapsto \theta_{\Lambda,P}(z),\]
where ${\rm Harm}_{\ell}(\R^d)\otimes_{\R} \C$ is the $\C$-vector space spanned by real harmonic polynomials.
When $\ell\ge1$, the image 
%Denote by $\varTheta_{\Lambda,\ell}$ the $\C$-vector space spanned by all weighted theta functions of weight $d/2+ \ell$;
\begin{equation*}
 {\rm Im}\, \vartheta_{\Lambda,\ell}= \langle \theta_{\Lambda,P}(z) \mid P\in {\rm Harm}_{\ell}(\R^d)\rangle_\C 
\end{equation*}
is a subspace of the $\C$-vector space $S_{d/2+\ell}(\Gamma_1(N))$ of cusp forms of weight $d/2+\ell$ {on} $\Gamma_1(N)$.

Fundamental results on the weighted theta functions for the $D_4$ lattice are summarized as follows.

\begin{prop}\label{prop:weight2}
For $\ell \ge1$, one has that $ {\rm Im}\, \vartheta_{D_4,\ell} \subset S_{2+\ell}(\Gamma_1(2))$.
 When $\ell=0$, we find that $\theta_{D_4,1}(z)=2E_2(2z) -E_2(z)=1+24q+24q^2+96q^3+24q^4+\cdots$,
 where
\[E_2(z):= 1-24 \sum_{m\ge1}\left(\sum_{d\mid m}d\right) q^m=1-24q-72q^2-96q^3-168q^4-144q^5+\cdots .\]
\end{prop}
\begin{proof}
The $D_4$ lattice is of level 2, so the first statement is a consequence of the classical results by Hecke and Schoenberg.
For the last statement, we note that the space $M_2(\Gamma_1(2))$ is 1-dimensional spanned by $2E_2(2z) -E_2(z)$ ({which is modular, even though the Eisenstein series $E_2(z)$ is not a modular form}).
Since ${\rm Im}\, \vartheta_{D_4,0}\subset M_2(\Gamma_1(2))$, $\theta_{D_4,1}$ is a constant multiple of $2E_2(2z) -E_2(z)$.
Comparing the constant term, we get the desired result.
\end{proof}

Let us prove the uniqueness of the $D_4$ lattice.

\begin{thm}\label{thm:uniqueness_lattice}
%Let $L$ be the set of all even integral lattice in $\R^4$ of level 2. % and discriminant 4.
For any even integral lattice $\Lambda \subset \R^4$ of level $2$, there exists an orthogonal transformation $\sigma \in O(\R^4)$ such that $\Lambda=\sigma\big(D_4\big) $.
\end{thm}
\begin{proof}
Since ${\rm Im}\, \vartheta_{\Lambda,0}\subset M_2(\Gamma_1(2))=\langle 2E_2(2z) -E_2(z)\rangle_\C$, we have $\theta_{\Lambda,1}(z)=2E_2(2z) -E_2(z)$.
This {together with Proposition \ref{prop:weight2}} implies $|\Lambda_{2m}| = |(D_4)_{2m}|$ for all $m\ge0$.
We first consider the case $\Lambda_2$.
Since $\Lambda$ is integral, using the Cauchy-Schwarz inequality, we see that $\langle \boldsymbol{x},\boldsymbol{y}\rangle \in \{0,\pm1,\pm2\}$ holds for any $\boldsymbol{x},\boldsymbol{y}\in \Lambda_2$.
Hence
\[ A\left(\frac{1}{\sqrt{2}} \Lambda_2\right) \subset \left\{ -1,-\frac12 ,0,\frac12\right\}.\]
Since a half set $X'$ of $\frac{1}{\sqrt{2}}\Lambda_2$ is a $(4,12,1/2)$ spherical code with $A(X')\subset\{-1/2,0,1/2\}$, by Theorem \ref{thm:upper_bound_design} and Lemma \ref{lem:half set of antipodal T-design}, the normalized set $\frac{1}{\sqrt{2}}\Lambda_2$ is an antipodal spherical $\{10,4,2\}$-design on $\mathbb{S}^3$ with 24 points. 
%From Corollary \ref{cor:equivalence_code_design}, the normalized set $\frac{1}{\sqrt{2}}\Lambda_2$ is an antipodal spherical $\{10,4,2\}$-design on $\mathbb{S}^3$ with 24 points. 
By Theorem \ref{thm:uniqueness}, there exists $\sigma\in O(\R^4)$ such that $\Lambda_2=\sigma\big({\mathbf D}_4\big)$.
Now let us consider the sublattice $\Lambda'$ of $\Lambda$ generated by $\Lambda_2$.
Since the $D_4$ lattice is generated by $\mathbf{D}_4$, we have $\Lambda'= \sigma\big(D_4\big)$.
Noting that the orthogonal transformation $\sigma$ preserves the inner product, we get 
\[ |(D_4)_{2m}|= |\Lambda'_{2m}| \le |\Lambda_{2m}| = |(D_4)_{2m}|\]
for all $m\ge0$.
Thus, $\Lambda'_{2m}=\Lambda_{2m}$ and hence $\Lambda'=\Lambda$, from which the desired result follows.
\end{proof}

It should be noted that Theorem \ref{thm:uniqueness_lattice} can be shown by the same method with the one described in Serre's book \cite[Chap.~V]{Serre}.
In this direction, we shall use ${\rm Aut}(\mathbf{D}_4)=W(\mathbf{F}_4)$ and a version of the Siegel mass formula \cite{Siegel}.

\begin{rem}\label{rem:another proof of the strength}
We briefly mention another proof of Proposition \ref{prop:D4 root}, using weighted theta function $\theta_{D_4,P}(z)$.
For this, we first notice that if $\theta_{D_4,P}(z)=0$ for all $P\in {\rm Harm}_\ell (\R^4)$, then every normalized $2m$-shell $\frac{1}{\sqrt{2m}} (D_4)_{2m}$ is a spherical $\{\ell\}$-design (this criterion was first used by Venkov \cite{Venkov84} in his design theoretical study on even unimodular lattices).
Therefore, it suffices to show that $ {\rm Im}\, \vartheta_{\Lambda,\ell}=0$ for $\ell \in \{10,4,2\}$, but this can be checked by a computer due to the fact that $M_k(\Gamma_1(2))$ is a finite dimensional vector space over $\C$ so that these modular forms are determined by first several Fourier coefficients (actually, we also need a list of harmonic polynomials of these degrees and the simple expression of the $2m$-shell of $D_4$).
Alternatively, the result would follow from the dimension formula for the space $S_k^{\rm new}(\Gamma_1(2))$ of newforms (see \cite{DS}), %\cite[Definition 5.8.1]{DS} for the definition of newforms) 
since we may have the equality ${\rm Im}\, \vartheta_{D_4,\ell} = S_{2+\ell}^{\rm new}(\Gamma_1(2))$ (this equality is folklore, but well known for the experts; consult \cite{BochererNebe,Eichler,HijikataSaito} for relevant materials).
\end{rem}

Combining the uniqueness of level 2 lattices (Theorem \ref{thm:uniqueness_lattice}) and Waldspurger's result \cite[Th\'{e}or\`{e}m 2']{Waldspurger79}, we can at least make sure that the inclusion
%\begin{equation*}\label{eq:inclusion_new} 
${\rm Im}\, \vartheta_{D_4,\ell} \supset S_{2+\ell}^{\rm new}(\Gamma_1(2))$ holds for any $\ell\ge1$.
%\end{equation*}
The first example of newforms {on} $\Gamma_1( 2)$ exists in weight 8 of the form
\[ \eta(z)^{8}\eta(2z)^8=q-8q^2+12q^3+64q^4-210q^5+\cdots,\] 
where $\eta(z)=q^{1/24}\prod_{n\ge1}(1-q^n)$ is the Dedekind eta function.
%We can obtain the inclusion
%\begin{equation*}\label{eq:inclusion_new} 
%${\rm Im}\, \vartheta_{D_4,\ell} \supset S_{2+\ell}^{\rm new}(\Gamma_1(2)),$
%\end{equation*}
%.
The above inclusion implies that there exists a harmonic polynomial $P\in {\rm Harm}_6(\R^4)$ such that $\theta_{D_4,P}(z)= \eta(z)^{8}\eta(2z)^8$.
We will give applications of this expression in the next section.

%For latter purpose, we recall that $f(z)=\sum_{m\ge1}c_f(m)q^m \in S_k(\Gamma_1(2))$ is a newform if it satisfies $c_f(1)=1$ and $T_nf(z) =c_f(n)f(z)$ for all $n\ge3$ odd, where $T_n$ denotes the Hecke operator, and is orthogonal to all elements in $\{g(z),g(2z) \mid g(z)\in S_k({\rm SL}_2(\Z))\}$ with respect to the Petersson inner product (see \cite[Definition 5.8.1]{DS}).
%Note that if $f \in S_k(\Gamma_1(2))$ is a newform, then $f$ is an eigenform for the Hecke operators $T_n$ for all $n$ (see \cite[Theorem 5.8.2]{DS}).
%In particular, one has 
%\begin{equation}\label{eq:t2_eigen}
%T_2f(z):=\sum_{m\ge1}c_f(2m)q^m=c_f(2)f(z),
%\end{equation}
%which implies $c_f(2m)=c_f(2)c_f(m)$ for all $m\ge1$.

%%%%%%%%%%%%%%%%%%%%%%%%%%%%%%%%%%
\section{Strength of spherical design}

In this section, we first prove Theorem \ref{thm:non-vanishing}, and then discuss the non-vanishing problem on the Fourier coefficients of the cusp form $\eta(z)^{8}\eta(2z)^8$.
For a finite set $X\subset \mathbb{S}^{d-1}$, we say that $T\subset \N$ is the harmonic strength of $X$ if $X$ is not a spherical $T'$-design for any $T\subsetneq T'\subset \N$.
We wish to determine the harmonic strength of the $2m$-shell of the $D_4$ lattice.
%We first indicate that this problem is intimately connected to the non-vanishing problem of the Fourier coefficients of the newform $\eta(z)^{8}\eta(2z)^8$.

\begin{thm}\label{thm:lehmer-type}
For $m\ge1$, the harmonic strength of $\frac{1}{\sqrt{2m}}\big(D_4\big)_{2m}$ contains $6$ if and only if $\tau_2(m)=0$, where $\sum_{m\ge1}\tau_2(m)q^m=\eta(z)^{8}\eta(2z)^8$.
\end{thm}
\begin{proof}
%We have proved in Proposition \ref{prop:D4 root} that the set $\frac{1}{\sqrt{2m}}\big(D_4\big)_{2m}$ is an antipodal spherical $\{10,4,2\}$-design.
%We only deal with whether $\frac{1}{\sqrt{2m}}\big(D_4\big)_{2m}$ is a $\{6\}$-design or not.
We first notice that by the representation theory, we have 
\begin{equation*}
{\rm Harm}_\ell (\R^4)=  {\rm Harm}_\ell (\R^4)^{W(\mathbf{F}_4)}\oplus  \{(1-\sigma^\ast)P\mid P\in {\rm Harm}_\ell (\R^4), \sigma\in W(\mathbf{F}_4)\}.
\end{equation*}
{For all $P\in {\rm Harm}_\ell (\R^4)$ and $\sigma\in W(\mathbf{F}_4)$, since the subgroup $W(\mathbf{F}_4)$ of $O(\R^4)$ acts on $D_4$, we have 
\[ \theta_{D_4,\sigma^\ast P}(z)=\sum_{\boldsymbol{x}\in D_4} P(\sigma(\boldsymbol{x}))q^{\frac{\langle \boldsymbol{x},\boldsymbol{x}\rangle}{2}} = \sum_{\boldsymbol{x}\in \sigma(D_4)} P(\boldsymbol{x})q^{\frac{\langle \boldsymbol{x},\boldsymbol{x}\rangle}{2} }= \theta_{\sigma(D_4),P}(z)=\theta_{D_4,P}(z).\]
Hence, $ \{(1-\sigma^\ast)P\mid P\in {\rm Harm}_\ell (\R^4), \sigma\in W(\mathbf{F}_4)\}$ is a subspace of ${\rm ker} \, \vartheta_{D_4,\ell}$.
This shows that} ${\rm Im}\, \vartheta_{D_4,\ell}={\rm Im}\, \vartheta_{D_4,\ell}\big|_{{\rm Harm}_\ell (\R^4)^{W(\mathbf{F}_4)}}$.

By \eqref{eq:har_molien}, the space ${\rm Harm}_6(\R^4)^{W(\mathbf{F}_4)}$ is the 1-dimensional subspace of ${\rm Harm}_6(\R^4)$ and its basis is given (see e.g., \cite[Section 5.1]{NozakiSawa}) by 
\begin{equation}\label{eq:pol_6}
\begin{aligned}
P_6(\boldsymbol{x}):=&p_6(x_1,x_2,x_3,x_4) \\
&-5 \big\{ x_1^4p_2(x_2,x_3,x_4)+x_1^2 p_4(x_2,x_3,x_4) +(x_2^4+x_3^2x_4^2) p_2(x_3,x_4) + x_2^2 p_4(x_3,x_4)\}\\
&+30\{x_1^2(x_2^2x_3^2+x_2^2x_4^2+x_3^2x_4^2)+x_2^2x_3^2x_4^2\},
\end{aligned}
\end{equation}
where $p_k(x_1,\ldots,x_d)=x_1^k+\cdots+x_d^k$.
From the above argument, we see that $\frac{1}{\sqrt{2m}}\big(D_4\big)_{2m}$ is a $\{6\}$-design if and only if $\sum_{\boldsymbol{x}\in (D_4)_{2m}} P_6(\boldsymbol{x})=0$.
Then, the result follows from the easily checked identity
\begin{equation}\label{eq:theta_eta}
\theta_{D_4,P_6}(z)=-192 \eta(z)^8\eta(2z)^8,
\end{equation}
where again, we have used the fact that the modular forms are determined by first several Fourier coefficients. 
\end{proof}

We remark that Theorem \ref{thm:lehmer-type} is an analogue to the one given by de la Harpe, Pache and Venkov \cite{HarpePache05,HarpePacheVenkov06}; They observed that the normalized $2m$-shell of the $E_8$ lattice is an antipodal spherical $8$-design if and only if $\tau(m)=0$, where $\tau(m)$ is the $m$th Fourier coefficient of the discriminant function $\Delta(z)=\eta(z)^{24}=\sum_{m\ge0}\tau(m)q^m\in S_{12}({\rm SL}_2(\Z))$.
The question of whether $\tau(m)\neq0$ holds for all $m\ge1$, posed by Lehmer \cite{Lehmer47}, is still far from being solved, so it is a common understanding that determining the (harmonic) strength for all shells of a given lattice is a hard problem.
A similar attempt for other lattices can be found in \cite{BM10,Pache05}.
In particular, Miezaki \cite{Miezaki13} obtained the harmonic strength for any shells of the square lattice $\Z^2$.
His result is extended by Pandey \cite{Pandey} to the rings of integers of imaginary quadratic fields over $\Q$ with class number 1.

Using Pari-GP \cite{PARI}, we have checked that $\tau_2(m)$ is non-zero up to $m\le 10^{8}$.
One would expect that the harmonic strength of the $2m$-shell of $D_4$ is given by $\{10,4,2\}$ for all $m\ge1$.
To give partial evidence, we consider the congruences of $\tau_2(m)$.

\begin{thm}\label{thm:tau2_congruence}
Let $\ell \in \{3,5\}$.
For any prime $p\ge3$, we have that
%\begin{equation}\label{eq:congruence}
\[\tau_2(p) \equiv p(p+1) \mod \ell.\]
%\end{equation}
\end{thm}

\begin{proof}
We use the harmonic polynomial $P_6$ defined in \eqref{eq:pol_6}.
For the case $\ell=3$, 
using $x^4\equiv x^2 \bmod 3$ for all $x\in \Z$, we get
\begin{align*} P_6(\boldsymbol{x})&\equiv x_1^4+x_2^4+x_3^4+x_4^4 + x_1^2p_2(x_2,x_3,x_4)+x_1^2 p_2(x_2,x_3,x_4) \\
&+(x_2^2+x_3^2x_4^2) p_2(x_3,x_4) + x_2^2 p_2(x_3,x_4)\\
&\equiv (x_1^2+\cdots+x_4^2)^2 \mod 3.
\end{align*}
This shows that $P_6(\boldsymbol{x})\equiv (2p)^2\bmod 3$ for all $\boldsymbol{x}\in (D_4)_{2p}$.
Since $|(D_4)_{2p}|=24(1+p)$ (see \eqref{eq:Jacobi}) is divisible by 3, from \eqref{eq:theta_eta} one obtains
\[ -64\tau_2(p) =\frac{1}{3}\sum_{\boldsymbol{x}\in (D_4)_{2p}} P_6(\boldsymbol{x}) \equiv \frac{1}{3}(2p)^2|(D_4)_{2p}|=32p^2 (1+p)\mod 3, \]
from which the case $\ell=3$ follows.
For the case $\ell=5$, notice that $x^6\equiv x^2\bmod 5$ holds for any $x\in \Z$.
We get
\[ P_6(\boldsymbol{x})\equiv x_1^2+x_2^2+x_3^2+x_4^2\mod 5,\]
and hence,
\[-192\tau_2(p) =\sum_{\boldsymbol{x}\in (D_4)_{2p}} P_6(\boldsymbol{x}) \equiv 2p|(D_4)_{2p}|=48p(1+p)\mod 5. \]
So we are done.
\end{proof}

{
\begin{cor}\label{cor:non-vanishing}
For any prime $p\not\equiv -1 \bmod 15$, we have that $\tau_2(p)\neq0$.
\end{cor}
\begin{proof}
This is immediate from Theorem \ref{thm:tau2_congruence}.
\end{proof}
}

Apart from non-vanishing of the $\tau_2$-function, we should mention that similar congruences to {Theorem \ref{thm:tau2_congruence}} have been established by many people since the time of Ramanujan (see e.g., \cite{BillereyMenares,DummiganFretwell,GabaPopa,KumarKumariMoreeSingh,Nikdelan}).
Our congruences could be a special case of them, but our proof is new.

\begin{rem}
In much the same way as \cite[Theorem 2]{Lehmer47}, we can prove the following statement: {The least $m_0$ such that $\tau_2(m_0)=0$, if exists, it will be an odd prime}.
Deligne's bound $|\tau_2(p)|\le 2p^{\frac{7}{2}}$ (see \cite[Theorem 8.2]{Deligne}) is one of key ingredients of the proof.
%Using the fact that the $\tau_2(m)$ is multiplicative, we may assume that $m_0$ is a prime power, say $m_0=p^a$.
%Since $\eta(z)^8\eta(2z)^9$ is an eigenfunction with respect to the Hecke operator $T_2=U_2$, we have $\tau_2(2^a)=\tau_2(2)^a=(-2)^{3a}\neq 0$.
%Hence $p$ should be odd.
%For $p\ge3$ prime, let $\alpha_p$ and $\beta_p$ be roots of the characteristic polynomial of the recursive formula 
%\[ \tau_2(p^a)= \tau_2(p)\tau_2(p^{a-1})-p^7 \tau_2(p^{a-2}) \quad (a\ge2).\]
%It holds that $\tau_2(p)=\alpha_p+\beta_p$ and $\ \alpha_p\beta_p=p^{11}$.
%Since $|\tau_2(p)|\le 2p^{\frac{7}{2}}$ (see \cite[Theorem 8.2]{Deligne}), there exists $\theta_p\in \R$ such that $\alpha_p=p^{\frac{7}{2}}e^{i\theta_p}$.
%With this, we have $\tau_2(p)=2p^{\frac{7}{2}}\cos \theta_p$ and by induction $\tau_2(p^a)=\frac{\alpha_p^{a+1}-\beta_p^{a+1}}{\alpha_p-\beta_p}$ for $a\ge2$.
%This shows
%\[\tau_2(p^a)= p^{\frac{7a}{2}}\frac{\sin (a+1)\theta_p}{\sin \theta_p}. \]
%Suppose $\tau_2(p)\neq0$.
%Then, since $\tau_2(p^a)=0$, it holds that $\sin(a+1)\theta_p=0$, i.e., $\theta_p\in \frac{\pi}{a+1}\Z$.
%Since $2\cos \theta_p$ is an algebraic integer, we get
%\[ \frac{\tau_2(p)^2}{p^7}=(2\cos \theta_p)^2 =v \in \Z_{\ge0}.\]
%By $|\cos \theta_p|\le1$, one has $v\le 4$.
%The right side of $\tau_2(p)^2=vp^7$ should be square, so $v $ is divisible by $p$.
%This happens only if $v=p=3$.
%But, since $\frac{\tau_2(3)}{3^7}=\frac{16}{243}\neq 3$, this never happens. 
%Thus, $\tau_2(p)=0$.
\end{rem}

%The data that support the findings of this study are available from the corresponding author, K.T. upon reasonable request.

%%%%%%%%%%%%%%%%%%%%%%%%%%%%%%%%%%%%%%%%%
\section*{Acknowledgments}
This work is partially supported by
JSPS KAKENHI Grant Number 19K03445, 20K03736, 20K14294 and 22K03402, and the Research Institute for Mathematical Sciences,
an International Joint Usage/Research Center located in Kyoto University.
The authors are grateful to Prof.~Eiichi Bannai, %Prof.~Tomonori Moriyama, 
Prof.~Ken Ono, Prof.~Siegfried B\"{o}cheler, Prof.~Jiacheng Xia and Prof.~Pieter Moree
% and Prof.~Don Zagier 
for valuable discussions, comments, suggestions, and corrections.
The authors also would like to extend their appreciation to the anonymous reviewers for their valuable feedback, which significantly contributed to the clarity and coherence of this paper.

%%%%%%%%%%%%%%%%%%%%%%%%%%%%%%%%%%%%%%%%%
%\section*{Statements and Declarations}

%We wish to confirm that there are no known conflicts of interest associated with this publication and there has been no significant financial support for this work that could have influenced its outcome.

\end{document}